\documentclass[reqno, twoside]{amsart}

\usepackage{amsmath}
\usepackage{amssymb}
\usepackage{amsthm}
\usepackage{graphicx}
\usepackage{subfig}
\usepackage{enumerate}
\usepackage{hyperref}

\hypersetup{%
pdftitle={Simple embeddings of rational homology balls and antiflips},%
pdfauthor={Heesang Park, Dongsoo Shin, Giancarlo Urz{\'u}a},%
pdfkeywords={antiflip, Mori sequence, rational homology ball},%
colorlinks=true,%
citecolor=black,%
filecolor=black,%
linkcolor=black,%
urlcolor=black,%
}

\usepackage{comment}
\usepackage{xspace}
\usepackage{mathtools}

\usepackage{mathptmx}       
\usepackage{helvet}         
\usepackage{courier}        

\usepackage{tikz}

\usetikzlibrary{patterns,decorations.pathreplacing}

\tikzset{bullet/.style={
shape = circle,fill = black, inner sep = 0pt, outer sep = 0pt, minimum size = 0.35em, line width = 0pt, draw=black!100}}

\tikzset{circle/.style={
shape = circle,fill = none, inner sep = 0pt, outer sep = 0pt, minimum size = 0.35em, line width = 1pt, draw=black!100}}

\tikzset{rectangle/.style={
shape = rectangle,fill = white, inner sep = 0pt, outer sep = 0pt, minimum size = 0.35em, line width = 0pt, draw=black!100}}

\tikzset{empty/.style={
shape = circle,fill = white, inner sep = 0pt, outer sep = 0pt, minimum size = 0.35em, line width = 0pt, draw=white!100}}

\tikzset{xmark/.style={
shape = x,fill = white, inner sep = 0pt, outer sep = 0pt, minimum size = 0em, line width = 0pt, draw=white!100}}

\tikzset{longrectangle/.style={
inner sep = 1em,
rectangle,
minimum size=1em,
very thick,
draw=black!100, 
}}

\tikzset{label distance=-0.15em}

\tikzset{font=\scriptsize}

\usepackage{tikz-cd}
\tikzcdset{every label/.append style = {font = \normalsize}}
\usepackage{tkz-graph}

\newtheorem{theorem}{Theorem}[section]
\newtheorem{lemma}[theorem]{Lemma}
\newtheorem{proposition}[theorem]{Proposition}
\newtheorem{corollary}[theorem]{Corollary}
\newtheorem*{corollarywon}{Corollary}

\newtheorem*{theoremknown}{Theorem}

\theoremstyle{definition}

\newtheorem{definition}[theorem]{Definition}
\newtheorem{remark}[theorem]{Remark}

\newtheorem{example}[theorem]{Example}

\numberwithin{equation}{section}

\def\sheaf#1{\ensuremath \mathcal#1}

\begin{document}

\title[Rational homology balls and antiflips]{Simple embeddings of rational homology balls and antiflips}

\author[H. Park]{Heesang Park}

\address{Department of Mathematics, Konkuk University, Seoul 05029, Republic of Korea \& Korea Institute for Advanced Study, Seoul 02455, Republic of Korea}

\email{HeesangPark@konkuk.ac.kr}

\author[D. Shin]{Dongsoo Shin}

\address{Department of Mathematics, Chungnam National University, Daejeon 34134 \& Korea Institute for Advanced Study, Seoul 02455, Republic of Korea}

\email{dsshin@cnu.ac.kr}

\author[Urz{\'u}a, Giancarlo]{Giancarlo Urz{\'u}a}

\address{Facultad de Matem{\'a}ticas, Pontificia Universidad Cat{\'o}lica de Chile, Santiago, Chile}

\email{urzua@mat.uc.cl}




\begin{abstract}
Let $V$ be a regular neighborhood of a negative chain of $2$-spheres (i.e. exceptional divisor of a cyclic quotient singularity), and let $B_{p,q}$ be a rational homology ball which is smoothly embedded in $V$. Assume that the embedding is simple, i.e. the corresponding rational blow-up can be obtained by just a sequence of ordinary blow-ups from $V$. Then we show that this simple embedding comes from the semi-stable minimal model program (MMP) for $3$-dimensional complex algebraic varieties under certain mild conditions. That is, one can find all simply embedded $B_{p,q}$'s in $V$ via a finite sequence of antiflips applied to a trivial family over a disk. As applications, simple embeddings are impossible for chains of $2$-spheres with self-intersections equal to $-2$. We also show that there are (infinitely many) pairs of disjoint $B_{p,q}$'s smoothly embedded in regular neighborhoods of (almost all) negative chains of $2$-spheres. Along the way, we describe how MMP gives (infinitely many) pairs of disjoint rational homology balls $B_{p,q}$ embedded in blown-up rational homology balls $B_{n,a} \# \overline{\mathbb{CP}^2}$ (via certain divisorial contractions), and in the Milnor fibers of certain cyclic quotient surface singularities. This generalizes results in Khodorovskiy-2014~\cite{Khodorovskiy-2012,Khodorovskiy-2014}, H. Park-J. Park-D. Shin-2016~\cite{PPS-2016}, Owens-2017~\cite{Owens-2017} by means of a uniform point of view.
\end{abstract}

\maketitle

\tableofcontents

\section{Introduction}

Let $0<q<p$ be coprime integers. A rational homology ball $B_{p,q}$ is the $4$-manifold obtained by attaching a $1$-handle and a single $2$-handle with framing $pq-1$ to $B^4$ by wrapping the attaching circle of the $2$-handle $p$-times around the $1$-handle with a $q/p$-twist. A rational homology ball $B_{p,q}$ may be also regarded as the Milnor fiber of the $\mathbb{Q}$-Gorenstein smoothing of a cyclic quotient surface singularity of type $\frac{1}{p^2}(1,pq-1)$ (i.e. a Wahl singularity). The lens space $L(p^2, pq-1)$ is the boundary of $B_{p,q}$. At the same time, the space $L(p^2,pq-1)$ is a boundary of a regular neighborhood $C_{p,q}$ of a linear chain of $2$-spheres $A_1, \dotsc, A_s$, whose dual graph is given by
\begin{equation}\label{equation:Cpq}
\begin{tikzpicture}
\node[bullet] (10) at (1,0) [label=above:{$-a_1$},label=below:{$A_1$}] {};
\node[bullet] (20) at (2,0) [label=above:{$-a_2$},label=below:{$A_2$}] {};

\node[empty] (250) at (2.5,0) [] {};
\node[empty] (30) at (3,0) [] {};

\node[bullet] (350) at (3.5,0) [label=above:{$-a_{s-1}$},label=below:{$A_{s-1}$}] {};
\node[bullet] (450) at (4.5,0) [label=above:{$-a_s$},label=below:{$A_s$}] {};

\draw [-] (10)--(20);
\draw [-] (20)--(250);
\draw [dotted] (20)--(350);
\draw [-] (30)--(350);
\draw [-] (350)--(450);
\end{tikzpicture}
\end{equation}
where $A_i \cdot A_i = - a_i \le -2$ with
\begin{equation*}
\frac{p^2}{pq-1} = a_1 - \cfrac{1}{a_2 - \cfrac{1}{\ddots - \cfrac{1}{a_s}}}
\end{equation*}
Hence one can \emph{rationally blow-up} $B_{p,q}$ which has been smoothly embedded in a smooth $4$-manifold $W$; that is, $\widetilde{W} = (W \setminus B_{p,q}) \cup_{L(p^2, pq-1)} C_{p,q}$.

A typical way to produce smooth embeddings of $B_{p,q}$'s is through the rational blow-down of a $C_{p,q}$ in a smooth $4$-manifold $X$. This  sometimes gives a $4$-manifold which is not trivially related to $X$. For example, from a rational elliptic fibration we can get Enriques  or Dolgachev surfaces, or from Lee-Park type of examples we can obtain surfaces of general type by applying rational blow-down to certain blow-ups of $\mathbb C \mathbb P^2$ (see e.g. Urz{\'u}a~\cite{Urzua-2016}). However, it is not easy to detect rational homology balls $B_{p,q}$ embedded in a given $4$-manifold $Z$ unless one knows a priori that $Z$ is obtained by performing a rational blow-down surgery.

Khodorovskiy~\cite{Khodorovskiy-2012,Khodorovskiy-2014} initiates the study of detecting $B_{p,q}$. The author defines \emph{simple} embeddings of $B_{p,q}$ in $V$ as the corresponding rational blow-up $W = (V \setminus B_{p,q}) \cup C_{p,q}$ is obtained by a sequence of ordinary blow-ups $\pi \colon W \to V$, that is, $W = V \sharp k \overline{\mathbb{CP}^2}$ for some $k \ge 1$. Using Kirby calculus, the author shows various instances of simple embeddings in regular neighborhoods of negative $2$-spheres.

\begin{theoremknown}[{Khodorovskiy~\cite[Theorem~1.2, 1.3]{Khodorovskiy-2014}}]\label{theoremknown:Khodorovskiy}
There is a rational homology ball $B_{p,1}$ simple embedded in a regular neighborhood $V_{-(p+1)}$ of a $2$-sphere with self-intersection number $-p-1$. In particular, for all odd $p > 1$, there exists a simple embedding of the rational homology balls $B_{p,1} \hookrightarrow V_{-4}$.
\end{theoremknown}

\begin{theoremknown}[{Khodorovskiy~\cite[Theorem~1.3]{Khodorovskiy-2014}}]\label{theoremknown:Khodorovskiy-sharp}
For all even $p \ge 2$, there exists a simple embedding of the rational homology balls $B_{p,1} \hookrightarrow B_{2,1} \sharp \overline{\mathbb{CP}^2}$.
\end{theoremknown}

Then in  PPS~\cite{PPS-2016}, the authors HP and DS with Jongil Park generalizes Khodorovskiy~\cite[Theorem~1.2]{Khodorovskiy-2014} as follows.

\begin{theoremknown}[{PPS~\cite[Theorem~4.1]{PPS-2016}}]\label{theoremknown:PPS}
Let $V$ be a plumbing $4$-manifold of the $\delta$-half linear chain corresponding to $p^2/(pq-1)$ with $1 \le q < p$. Then there is an embedded rational homology ball $B_{p,q}$ in $V$.
\end{theoremknown}

Here the \emph{$\delta$-half linear chain} associated to $(p,q)$ is, roughly speaking, a half of the negative-definite plumbing graph associated to $C_{p,q}$ (see PPS~\cite{PPS-2016}). The main tools of PPS~\cite{PPS-2016} are techniques from explicit semi-stable minimal model program for $3$-dimensional complex algebraic varieties. Using the MMP for 3-folds, they also prove again the result above of Khodorovskiy~\cite[Theorem~1.3]{Khodorovskiy-2014}.

Recently Owens~\cite{Owens-2017} gives topological proofs for the results in Khodorovskiy~\cite{Khodorovskiy-2014} and PPS~\cite{PPS-2016} by using manipulations of knot with bands diagrams representing properly embedded surfaces in $B^4$, which also leads the following new embeddings.

\begin{theoremknown}[{Owens~\cite[Theorem~4]{Owens-2017}}]\label{theoremknown:Owens-sharp}
There is a simple embedding $B_{p^2,p-1} \hookrightarrow B_{p,1} \sharp \overline{\mathbb{CP}^2}$.
\end{theoremknown}

\begin{theoremknown}[{Owens~\cite[Theorem~5]{Owens-2017}}]\label{theoremknown:Owens-CP2bar}
Let $F(n)$ is the $n$-th Fibonacci number. Then there is a smooth embedding $B_{F(2n+2),F(2n)} \hookrightarrow \overline{\mathbb{CP}^2}$.
\end{theoremknown}

Owens~\cite{Owens-2017} found also new embeddings into $\mathbb{CP}^2$. These embeddings cannot be symplectic by the result of Evans-Smith~\cite{Evans-Smith-2018} because they proved that all symplectic embeddings into $\mathbb{CP}^2$ are the ones coming from algebraic geometry via the classification of Hacking-Prokhorov~\cite{Hacking-Prokhorov-2010}. We remark that these complex embeddings are also related to MMP via a blown-up family (see Urz{\'u}a~\cite[Theorem 3.1]{Urzua-2016a}).

In this paper we show that all simple embeddings in neighborhoods of linear chains of negative $2$-spheres come from the semi-stable MMP for $3$-dimensional complex algebraic varieties under certain mild conditions. Explicitly:

\begin{theorem}[Theorem~\ref{theorem:main-derived-from-MMP} for details]\label{maintheorem-simple-via-MMP}
If $V$ is a regular neighborhood of a linear chain of 2-spheres which is diffeomorphic to the minimal resolution of a germ of a cyclic quotient surface singularity,
then every plainly simple embedding (Definition~\ref{definition:plainly-simple}) of $B_{p,q}$ in $V$ comes from a finite sequence of antiflips applied to a trivial family over a disk.
\end{theorem}

We remark that for a one parameter $\mathbb Q$-Gorenstein smoothing of a complex projective surface with only Wahl singularities, we have that if the general fiber is birational to the special fiber, then the $\mathbb Q$-Gorenstein smoothing comes from a smooth deformation after applying MMP (see Urz{\'u}a~\cite[Section 3]{Urzua-2016a}). In this ``birational $\mathbb Q$-Gorenstein smoothing" situation all the corresponding embeddings of $B_{p,q}$'s are simple.

\begin{corollarywon}[Corollary~\ref{corollary:No-in-An}] \label{coro}
There are no plainly simple embedded rational homology balls in a regular neighborhood of a linear chain $\Gamma$ of consisting of only $2$-spheres with self-intersections equal to $-2$.
\end{corollarywon}

So in principle one may find all plainly simple embedded $B_{p,q}$ in $V$ if one can handle the combinatorics of all possible mk1A's or mk2A's anti-flips (see Section~\ref{seciton:Mori-sequences}) for a given trivial family with $V$ as its central fiber, which ends up in a $C_{p,q}$ rational blow-down of a blow-up of $V$. But describing this explicitly is a complicated combinatorial problem in general.

However, we generalize theorems of Khodorovskiy~\cite{Khodorovskiy-2014}, PPS~\cite{PPS-2016}, Owens~\cite{Owens-2017} above by means of MMP in the following way. Let $V$ be a regular neighborhood of a linear chain of $2$-spheres $E_1 \cup \dotsb \cup E_t$ whose dual graph $\Gamma$ is given by
\begin{equation*}
\begin{tikzpicture}[scale=0.75]
\node[bullet] (20) at (2,0) [label=above:{$-e_1$},label=below:{$E_1$}] {};

\node[empty] (250) at (2.5,0) [] {};
\node[empty] (30) at (3,0) [] {};

\node[bullet] (350) at (3.5,0) [label=above:{$-e_t$},label=below:{$E_t$}] {};

\draw [-] (20)--(250);
\draw [dotted] (20)--(350);
\draw [-] (30)--(350);

\end{tikzpicture}
\end{equation*}
where $E_i \cdot E_i = - e_i \le -2$ for all $i=1,\dotsc,t$. We prove the following.

\begin{theorem}[Theorem~\ref{theorem:main-linear}, Proposition~\ref{proposition:simple-embedding}]\label{maintheorem-existence}
If $e_t \ge 3$, then there are pairs of disjoint $B_{p,q}$'s for various $p, q$ smoothly embedded in $V$. Furthermore, there are infinitely many such pairs unless $e_1=\dotsb=e_{t-1}=2$ and $e_t=3$. All the embeddings are simple.
\end{theorem}

Here $p$ and $q$ are explicitly determined by numerical data of the graph $\Gamma$; See Theorem~\ref{theorem:main-linear} for details. In particular, we show that there are infinitely many pairs of disjoint $B_{p,q}$'s in $V_{-n}$ for any $n \ge 4$ (Corollary~\ref{corollary:V_-n}).

Another simple consequences of the MMP procedure give the following two theorems.

\begin{theorem}[cf.~Theorem~\ref{theorem:main-sharp}, Proposition~\ref{proposition:simple-embedding-sharp}]\label{tmaintheorem-Bsharp}
For any integers $n, a$ with $n > a \ge 1$ and $(n,a)=1$, there are infinitely many pairs of disjoint rational homology balls smoothly embedded in $B_{n,a} \sharp \overline{\mathbb{CP}^2}$, where one of them is $(B_{n^2,na-1}, B_{n,a})$. All of the embeddings except $B_{n,a} \hookrightarrow B_{n,a} \sharp \overline{\mathbb{CP}^2}$ are proved to be simple.
\end{theorem}

\begin{theorem}[cf.~Theorem~\ref{theorem:main-2-handlebody}]\label{maintheorem-handlebody}
Let $(Q \in Y)$ be a cyclic quotient surface singularity which admits an extremal P-resolution $f^+ \colon X^+ \to Y$. Then there are infinitely many pairs of disjoint rational homology balls smoothly embedded in the Milnor fiber of $(Q \in Y)$ associated to the P-resolution $X^+ \to Y$ and all of the embeddings are non-simple.
\end{theorem}

Theorem~\ref{maintheorem-handlebody} may be regarded as a generalization of Park-Shin~\cite[Theorem~1.3]{PS-2017}:

\begin{theoremknown}[{Park-Shin~\cite[Theorem~1.3]{PS-2017}}]\label{theoremknown:PS}
There is a smoothly embedded $B_{p,1}$ in a non-simple way inside a $2$-handlebody corresponding to a certain twisted torus knot.
\end{theoremknown}

Since any Milnor fibers are Stein, they can be described as $2$-handlebodies over Legendrian links. It would be an intriguing problem to find twisted torus knots associated to Milnor fibers, which reinterprets the above theorems in the language of the MMP for $3$-folds and $2$-handlebodies, vice versa. We leave it for the future research.

The paper is organized as follows. In Section~\ref{section:preliminaries} we recall some definitions, notations, and key facts that will be used in the proofs in the next sections. This includes as main tool the explicit semi-stable minimal model program for 3-dimensional algebraic varieties as in HTU~\cite{HTU-2017}. Our main reference will be Urz{\'u}a~\cite[Section~2]{Urzua-2016}.
We prove that every simple embedding comes from MMP (Theorem~\ref{maintheorem-simple-via-MMP}) in Section~\ref{section:simple-embeddings-MMP}, via an argument involving semi-stable MMP and Milnor numbers of smoothings of cyclic quotient surface singularities. The proof of existence of infinitely many simple embeddings (Theorem~\ref{maintheorem-handlebody}) will be given in Section~\ref{section:Rational-homology-balls}, and involves a key combinatorial fact on Wahl singularities plus the universal family of extremal neighborhoods in HTU~\cite{HTU-2017}. We provide non-simple embeddings in Section~\ref{section:other-embedding} using directly semi-stable MMP.

\subsection*{Acknowledgements}

Heesang Park was supported by Basic Science Research Program through the National Research Foundation of Korea (NRF) funded by the Ministry of Education: NRF-2018R1D1A1B07042134. This paper was written as part of Konkuk University's research support program for its faculty on sabbatical leave in 2019. Dongsoo Shin was supported by Samsung Science and Technology Foundation under Project Number SSTF-BA1402-03. Giancarlo Urz{\'u}a was supported by the FONDECYT regular grant 1190066. A part of this work was done when Heesang Park and Dongsoo Shin were on sabbatical leave at Korea Institute for Advanced Study. They would like to thank KIAS for warm hospitality and financial support.

\section{Preliminaries}
\label{section:preliminaries}

We briefly recall some basics on the semi-stable MMP and the related topics used in the present paper. For more details we will refer to Urz{\'u}a~\cite[Section~2]{Urzua-2016}. We would like to point out that Evans-Urz{\'u}a~\cite{Evans-Urzua-2018} gives a symplectic topology guide to this MMP.

\subsection{Singularities and their $P$-resolutions}

A \emph{quotient surface singularity} is a germ of the quotient $\mathbb{C}^2/G$ by a (small) finite subgroup $G \le GL(2,\mathbb{C})$. In particular, a \emph{cyclic quotient surface singularity} of type $\frac{1}{n}(1,a)$ ($n > a \ge 1$, $(n,a)=1$) is a germ of a quotient $\mathbb{C}^2/\mu_n$ of $\mathbb{C}^2$ by a multiplicative group $\mu_n =\langle \zeta \rangle$, $\zeta = \exp(2 \pi \sqrt{-1} /n)$, via the action $\zeta \cdot (x, y) = (\zeta x, \zeta^a y)$. Let $E_1, \dotsc, E_s$ be the exceptional curves of the minimal resolution of a cyclic quotient surface singularity of type $\frac{1}{n}(1,a)$ with $E_j \cong \mathbb{CP}^1$ and $E_j^2=-e_j \le -2$ for all $j$. Then the self-intersection number $e_i$'s are determined by the \emph{Hirzebruch-Jung continued fraction} of $n/a$, that is,
\begin{equation*}
\frac{n}{a} = e_1 - \cfrac{1}{e_2 - \cfrac{1}{\ddots - \cfrac{1}{e_s}}} =: [e_1, e_2, \dotsc, e_s]
\end{equation*}
and the dual graph of $E_1, \dotsc, E_s$ is a \emph{linear chain of rational curves}
\begin{equation*}
\begin{tikzpicture}[scale=0.75]
\node[bullet] (10) at (1,0) [label=above:{$-e_1$},label=below:{$E_1$}] {};
\node[bullet] (20) at (2,0) [label=above:{$-e_2$},label=below:{$E_2$}] {};

\node[empty] (250) at (2.5,0) [] {};
\node[empty] (30) at (3,0) [] {};

\node[bullet] (350) at (3.5,0) [label=above:{$-e_s$},label=below:{$E_s$}] {};

\draw [-] (10)--(20);
\draw [-] (20)--(250);
\draw [dotted] (20)--(350);
\draw [-] (30)--(350);
\end{tikzpicture}
\end{equation*}

Notice that two Hirzebruch-Jung continued fractions $n/a=[e_1,\dotsc,e_s]$ and $n/(n-a)=[b_1, \dotsc, b_e]$ are \emph{dual} to each other, this is, we have
\begin{equation*}
[e_1,\dotsc,e_s,1,b_e,\dotsc,b_1]=0,
\end{equation*}
which implies that the linear chain of rational curves
\begin{tikzpicture}[scale=0.75]
\node[bullet] (20) at (2,0) [label=above:{$-e_1$}] {};

\node[empty] (250) at (2.5,0) [] {};
\node[empty] (30) at (3,0) [] {};

\node[bullet] (350) at (3.5,0) [label=above:{$-e_s$}] {};
\node[bullet] (450) at (4.5,0) [label=above:{$-1$}] {};

\node[bullet] (550) at (5.5,0) [label=above:{$-b_e$}] {};

\node[empty] (60) at (6,0) [] {};
\node[empty] (650) at (6.5,0) [] {};

\node[bullet] (70) at (7,0) [label=above:{$-b_1$}] {};

\draw [-] (20)--(250);
\draw [dotted] (20)--(350);
\draw [-] (30)--(350);

\draw [-] (350)--(450);
\draw [-] (450)--(550);

\draw [-] (550)--(60);
\draw [dotted] (550)--(70);
\draw [-] (650)--(70);
\end{tikzpicture}
is blown down to
\begin{tikzpicture}
\node[bullet] (00) at (0,0) [label=above:{0}] {};
\end{tikzpicture}.

Let $0 \in \mathbb D$ be the an analytic germ of a smooth curve (i.e. an arbitrarily small disk). A one-parameter deformation $(Y \subset \mathcal{Y}) \to (0 \in \mathbb{D})$ of a normal surface $Y$ is a \emph{smoothing} of $Y$ if a general fiber is smooth. It is \emph{$\mathbb{Q}$-Gorenstein} if $K_{\mathcal{Y}}$ is $\mathbb{Q}$-Cartier. A normal surface singularity is of \emph{class $T$} if it is a quotient surface singularity and admits a $\mathbb{Q}$-Gorenstein one-parameter smoothing KSB~\cite[Definition~3.7]{Kollar-Shepherd-Barron-1988}. Singularities of class $T$ are determined completely KSB~\cite[Proposition~3.10]{Kollar-Shepherd-Barron-1988}, and they can be rational double points or a cyclic quotient surface singularities of type $\frac{1}{dn^2}(1,dna-1)$ for $d \ge 1$, $n > a \ge 1$, $(n,a)=1$. Non rational double point T-singularities have a very particular minimal resolution KSB~\cite[Proposition~3.11]{Kollar-Shepherd-Barron-1988}.

A \emph{Wahl singularity defined by a pair $(n,a)$} is a cyclic quotient surface singularity of type $\frac{1}{n^2}(1,na-1)$ for $n > a \ge 1$, $(n,a)=1$. Every Wahl singularity admits a $\mathbb{Q}$-Gorenstein smoothing whose Milnor fiber is diffeomorphic to a rational homology ball $B_{p,q}$; Wahl~\cite[Examples~(5.9.1)]{Wahl-1981}.

Wahl singularities satisfy the following particular property, which will be key for our result on rational homology balls embedded in negative linear chains.

\begin{corollary}
Let $0<a<n$ be integers with gcd $(n,a)=1$. Let $n/a=[a_1,\dotsc,a_p]$ and $n/(n-a)=[b_1, \dotsc, b_q]$. Then

\begin{equation*}
\frac{n^2}{na-1} = [a_1, \dotsc, a_{p-1}, a_p+b_q, b_{q-1}, \dotsc, b_1].
\end{equation*}
\end{corollary}

\begin{proof}
See Owens~\cite[Lemma~3.1]{Owens-2017} for example. See also Hacking-Prokhorov~\cite{Hacking-Prokhorov-2010} in more generality.
\end{proof}

\begin{corollary}\label{corollray:[a]->[T]-[a]}
Let
\begin{tikzpicture}[scale=0.75]
\node[bullet] (20) at (2,0) [label=above:{$-a_1$}] {};

\node[empty] (250) at (2.5,0) [] {};
\node[empty] (30) at (3,0) [] {};

\node[bullet] (350) at (3.5,0) [label=above:{$-a_p$}] {};

\draw [-] (20)--(250);
\draw [dotted] (20)--(350);
\draw [-] (30)--(350);
\end{tikzpicture}
be a linear chain with $n/a=[a_1,\dotsc,a_p]$. Blowing up the vertex
\begin{tikzpicture}[scale=0.75]
\node[bullet] (350) at (3.5,0) [label=above:{$-a_p$}] {};
\end{tikzpicture}
appropriately, we have the following linear chain
\begin{equation*}
\begin{tikzpicture}
\node[bullet] (20) at (2,0) [label=above:{$-a_1$}] {};

\node[empty] (250) at (2.5,0) [] {};
\node[empty] (30) at (3,0) [] {};

\node[bullet] (350) at (3.5,0) [label=above:{$-a_{p-1}$}] {};
\node[bullet] (450) at (4.5,0) [label=above:{$-a_p-b_q$}] {};

\node[bullet] (550) at (5.5,0) [label=above:{$-b_{q-1}$}] {};

\node[empty] (60) at (6,0) [] {};
\node[empty] (650) at (6.5,0) [] {};

\node[bullet] (70) at (7,0) [label=above:{$-b_1$}] {};

\node[bullet] (80) at (8,0) [label=above:{$-1$}] {};

\node[bullet] (90) at (9,0) [label=above:{$-a_1$}] {};

\node[empty] (950) at (9.5,0) [] {};
\node[empty] (100) at (10,0) [] {};

\node[bullet] (1050) at (10.5,0) [label=above:{$-a_p$}] {};

\draw [-] (20)--(250);
\draw [dotted] (20)--(350);
\draw [-] (30)--(350);

\draw [-] (350)--(450);
\draw [-] (450)--(550);

\draw [-] (550)--(60);
\draw [dotted] (550)--(70);
\draw [-] (650)--(70);
\draw [-] (70)--(80);
\draw [-] (80)--(90);

\draw [-] (90)--(950);
\draw [dotted] (90)--(1050);
\draw [-] (100)--(1050);
\end{tikzpicture}
\end{equation*}
\end{corollary}

\begin{proof}
This follows easily from the fact that $[b_q,\dotsc,b_1,1,a_1,\dotsc,a_p]=0$.
\end{proof}

Koll{\'a}r and Shepherd-Barron~\cite[Section~3]{Kollar-Shepherd-Barron-1988} show that any deformations of a quotient surface singularity can be induced deformations from certain special partial resolutions (called \emph{$P$-resolutions}; Definition~\ref{definition:P-resolution}) with only singularities of class $T$.

\begin{definition}\label{definition:P-resolution}
A \emph{$P$-resolution} of (a germ of) a quotient surface singularity $(Q \in Y)$ is a partial resolution $g \colon Z \to Y$ of $Q \in Y$ such that $Z$ has only singularities of class $T$ and $K_Z$ is ample relative to $g$.
\end{definition}

\begin{proposition}[{cf. KSB~\cite[Theorem~3.9]{Kollar-Shepherd-Barron-1988}}]\label{proposition:P-resolution}
For any smoothing $\pi \colon \mathcal{Y} \to \mathbb{D}$ of a quotient surface singularity $(Q \in Y)$, there is a $P$-resolution $g \colon Z \to Y$ and a $\mathbb{Q}$-Gorenstein smoothing $\psi \colon \mathcal{Z} \to \mathbb{D}$ inducing $\mathbb{Q}$-Gorenstein smoothings of each singularities of class $T$ in $Z$ such that the smoothing $\mathcal{Z} \to \mathbb{D}$ \emph{blows down} to $\mathcal{Y} \to \mathbb{D}$, that is, there is a morphism $G \colon \mathcal{Z} \to \mathcal{Y}$ satisfying $\psi = \pi \circ G$.
\end{proposition}

\subsection{The semi-stable MMP and flips}

We will use in our proofs the explicit semi-stable MMP for families of surfaces over $\mathbb D$ which is described in HTU~\cite{HTU-2017}. We will refer to  Urz{\'u}a~\cite[Section~2]{Urzua-2016} as a quick concise reference.

The notation for a three dimensional \emph{extremal neighborhood} will be $$(C \subset X \subset \mathcal{X}) \to (Q \in Y \subset \mathcal{Y}),$$ where $C$ will be always a $\mathbb P^1$, $X$ a surface with at most two Wahl singularities such that $C^2<0$ and $C \cdot K_X <0$, and $C$ contracts to a cyclic quotient singularity $(Q \in Y)$. An extremal neighborhood is said to be \emph{flipping} if the exceptional set of $F$ is $C$. On the other hand, if it is not flipping, then the exceptional set of $F$ is of dimension $2$. In this case we call it a \emph{divisorial} extremal neighborhood. These flipping or divisorial contractions will be always mk1A or mk2A; see Urz{\'u}a~\cite[Section~2]{Urzua-2016} for definition and numerical description. In these cases, for any flipping extremal neighborhood we always have a proper birational morphism called \emph{flip} $$(C^+ \subset X^+ \subset \mathcal{X}^+) \to (Q \in Y \subset \mathcal{Y})$$ where $X^+ \to Y$ is an extremal P-resolution of $(Q \in Y)$. We also refer the induced birational map $\mathcal{X} \dashrightarrow \mathcal{X}^+$ as \emph{flip}. We recall that an \emph{extremal $P$-resolution $f^+ \colon X^+ \to Y$} of a cyclic quotient singularity $(Q \in Y)$ is a $P$-resolution such that $C^+=(f^+)^{-1}(Q)$ is an irreducible curve with only Wahl singularities.

A useful fact is the following.

\begin{proposition}[{see e.g. Urz{\'u}a~\cite[Proposition~2.8]{Urzua-2016}}]
\label{proposition:divisorial-contraction-for-mk12A}
If an mk1A or mk2A is divisorial, then $(Q \in Y)$ is a Wahl singularity. And the divisorial contraction $F \colon \mathcal{X} \to \mathcal{Y}$ induces the blowing down of a $(-1)$-curve between the smooth fibers of $\mathcal{X} \to \mathbb{D}$ and $\mathcal{Y} \to \mathbb{D}$.
\end{proposition}

We will use the notation in Urz{\'u}a~\cite[Subsection 2.1]{Urzua-2016} for the extremal neighborhoods mk1a and mk2A. In particular a mk1A will be presented by $[e_1, \dotsc, \overline{e_i}, \dotsc, e_s]$ where the Wahl singularity is $\frac{m^2}{ma-1}=[e_1, \dotsc, e_i, \dotsc, e_s]$. For a mk2A we will write $$[f_{s_2}, \dotsc, f_1] - [e_1, \dotsc, e_{s_1}], $$ where $\frac{m_2^2}{m_2a_2-1}=[f_1, \dotsc, f_{s_2}]$ and $\frac{m_1^2}{m_1a_1-1}=[e_1, \dotsc, e_{s_1}]$ correspond to the two Wahl singularities. For the flip, the extremal P-resolution will be represented by $$[f_{s_2}, \dotsc, f_1] - c - [e_1, \dotsc, e_{s_1}],$$ where $\frac{{m'}_2^2}{{m'}_2 {a'}_2-1}=[f_1, \dotsc, f_{s_2}]$ and $\frac{{m'}_1^2}{{m'}_1{a'}_1-1}=[e_1, \dotsc, e_{s_1}]$ are the Wahl singularities. Precise terminology and numerics can be found in Urz{\'u}a~\cite[Subsection 2.1]{Urzua-2016}.

\subsection{Mori sequences}
\label{seciton:Mori-sequences}

It is proved in HTU~\cite{HTU-2017} that for a fixed extremal P-resolution $X^+$, we have a family of flipping extremal neighborhoods of type mk1A and mk2A whose flip is a $\mathbb Q$-Gorenstein smoothing of $X^+$. It is also proved that for a fixed Wahl singularity $(Q \in Y)$, we have a family of divisorial extremal neighborhoods of type
mk1A and mk2A whose divisorial contraction is a $\mathbb Q$-Gorenstein smoothing of $(Q \in Y)$. These families can be described numerically via \emph{Mori sequences}. In these families there is an invariant $\delta \in \mathbb Z_{>0}$ which can be read from the intersection of the flipping curve with the canonical class. If $\delta=1$, then the Mori sequence is finite. Otherwise the Mori sequence is infinite.

We summarize briefly HTU~\cite[Subsection~3.3]{HTU-2017}. See also  Urz{\'u}a~\cite[Section~2]{Urzua-2016}.

Let $\mathbb{E}_1$ be an mk2A with Wahl singularities defined by pairs $(m_1, a_1)$ and $(m_2, a_2)$ with $m_2 > m_1$, that is, an mk2A with two Wahl singularities of type $\frac{1}{m_1^2}(1, m_1a_1-1)$ and $\frac{1}{m_2^2}(1, m_2a_2-1)$. Here we also allow the mk1A for special case $m_1=a_1=1$. Let $\delta=m_2a_1+m_1a_2-m_1m_2$. Since $K_X \cdot C = -\delta/(m_1m_2) < 0$, we have $\delta \ge 1$.

Assume that $\delta m_1 - m_2 \le 0$ from now on. Such an mk2A is called an \emph{initial} mk2A in HTU~\cite{HTU-2017}. We first define two sequences $d(i)$ and $c(i)$ (called the \emph{Mori recursions}) as follows: For $i \ge 2$,
\begin{align*}
&d(1)=m_1, \quad d(2)=m_2, \quad d(i+1)=\delta d(i)-d(i-1), \\
&c(1)=a_1, \quad c(2)=m_2-a_2, \quad c(i+1)=\delta c(i)-c(i-1).
\end{align*}

\begin{definition}[Mori sequence for $\delta > 1$]
If $\delta > 1$, then a \emph{Mori sequence} for an initial mk2A $\mathbb{E}_1$ is a sequence of mk2A's $\mathbb{E}_i$ with Wahl singularities defined by the pairs $(m_i, a_i)$ and $(m_{i+1}, a_{i+1})$ where
\begin{equation*}
m_{i+1}=d(i+1), a_{i+1}=d(i+1)-c(i+1), \quad m_i=d(i), a_i=c(i).
\end{equation*}
Notice that $m_{i+1} > m_i$.
\end{definition}

\begin{definition}[Mori sequence for $\delta=1$]
If $\delta=1$, then a \emph{Mori sequence} for an initial mk2A $\mathbb{E}_1$ consists of one more mk2A $\mathbb{E}_2$ defined by the pairs $(m_2, a_2)$ and $(m_3, a_3)$, where $m_2=d(2)$, $a_2=c(2)$, and $m_3=d(2)-d(1)$, $a_3=d(2)-d(1)+c(1)-c(2)$.
\end{definition}

\begin{proposition}[{HTU~\cite[Proposition~3.15, Theorem~3.20]{HTU-2017}}]\label{proposition:flip-for-Mori-sequence}
If $\delta m_1 - m_2 < 0$, then the mk2A $\mathbb{E}_i$'s ($i \ge 1$) in the Mori sequence of an initial mk2A $\mathbb{E}_1$ are of flipping type sharing the same $\delta$, $\Omega$, $\Delta$ associated to $\mathbb{E}_1$ and, after flipping each $\mathbb{E}_i$, they have the same extremal $P$-resolution $X^+$ admitting two Wahl singularities defined by the pairs $(m_1', a_1')$ and $(m_2', a_2')$ where $m_2'=m_1$, $a_2'=m_1-a_1$, and $m_1'=m_2-\delta m_1$, $a_1' \equiv m_2-a_2 -\delta a_1 \pmod{m_1'}$. Here, in case of $m_1=a_1=1$, we set $a_2'=1$.
\end{proposition}

\begin{proposition}[{HTU~\cite[Proposition~3.13]{HTU-2017}}]
If $\delta m_1 - m_2=0$, then $\mathbb{E}_i$'s are of divisorial type with $m_1=\delta$, $m_2=\delta^2$, and $a_2=\delta^2-(\delta a_1 -1)$. By Proposition~\ref{proposition:divisorial-contraction-for-mk12A}, the contraction $\mathcal{X}_i \to \mathcal{Y}$ over $\mathbb{D}$ corresponding to $\mathbb{E}_i$ induces the blow down $X_{i,t} \to Y_t$ over $0 \neq t \in \mathbb{D}$ of a $(-1)$-curve between smooth fibers.
\end{proposition}

\begin{example}[{Mori sequence of a flipping family for $\delta > 1$; cf. Urz{\'u}a~\cite[Example~2.14]{Urzua-2016}}]
\label{example:mori-sequence-flipping-delta>1}
Let $(Q \in Y)$ be the cyclic quotient surface singularity of type $\frac{1}{11}(1,3)$. Consider the extremal $P$-resolution $X^+ \to Y$ defined by $[4]-3$. So we have $m_1'=1$, $a_1'=1$, and $m_2'=2$, $a_2'=1$, and $\delta=3$. There are two corresponding initial flipping mk2A's: One is defined by the pairs $(1,1)$ and $(5,3)$ and the other is defined by $(2,1)$ and $(7,5)$. Therefore the first three entries of each Mori sequences are defined by the following pairs:
\begin{align*}
&\mathbb{E}_1: (1,1), (5,3),\
\mathbb{E}_2: (5,2), (14,9),\
\mathbb{E}_3: (14,5), (37,24), \dotsc \\
&\mathbb{E}_1: (2,1), (7,5),\
\mathbb{E}_2: (7,2), (19,14),\
\mathbb{E}_3: (19,5), (50,37), \dotsc
\end{align*}
Notice that two Wahl singularities defined by $(n,a)$ and $(n,n-a)$ represent the same singularity. So numerical data of any mk1A and any mk2A associated to $[4]-3$ can be read from
\begin{equation*}
\dotsb-[3,7,5,\overline{2},2,2,2,3,2]-[3,7,2,\overline{2},3,2]-[3,5,\overline{2}]-\varnothing
\end{equation*}
and
\begin{equation*}
\dotsb-[4,7,5,\overline{2},2,2,2,3,2,2]-[4,7,2,\overline{2},3,2,2]-[4,5,\overline{2},2]-[4],
\end{equation*}
where $\varnothing$ represents a smooth point (that is, the case of $m_1=a_1=1$ in the initial mk2A). These two Mori sequences provide the numerical data of all flipping extremal neighborhoods whose flip have $[4]-3$ as central fiber. For example, the mk2A's $[3,7,2,2,3,2]-[3,5,2]$ and $[4,7,2,2,3,2,2]-[4,5,2,2]$ or the mk1A $[4,7,5,\overline{2},2,2,2,3,2,2]$ have flips whose central fiber is $[4]-3$.
\end{example}

\begin{example}[Mori sequence of a flipping family for $\delta=1$]
\label{example:mori-sequence-flipping-delta=1}
Let $(Q \in Y)$ be the cyclic quotient surface singularity of type $\frac{1}{13}(1,3)$. Consider the extremal $P$-resolution $X^+ \to Y$ defined by $[5,2]-2$. Then $\delta=1$. There are two initial flipping mk2A's; but, their Mori sequences coincide. The numerical data of any mk1A and any mk2A associated to $[5,2]-2$ can be read from
\begin{equation*}
[2,5]-[2,2,\overline{6}]-\varnothing.
\end{equation*}
\end{example}

\begin{example}[{Mori sequence of a divisorial family; Urz{\'u}a~\cite[Example~2.13]{Urzua-2016}}]
\label{example:mori-sequence-divisorial}
Let $(Q \in \mathcal{Y})$ be the Wahl singularity $\frac{1}{4}(1,1)$. Then the numerical data of any mk1A and mk2A of divisorial type associated to $(Q \in \mathcal{Y})$ is as follows:
\begin{equation*}
\dotsb-[10,\overline{2},2,2,2,2,2]-[8,\overline{2},2,2,2]-[6,\overline{2},2]-[4].
\end{equation*}

\end{example}

A flip which appears frequently in this paper is the following.

\begin{proposition}[cf.~{Urz{\'u}a~\cite[Proposition~2.15]{Urzua-2016}}]
\label{proposition:usual-flip}
Let $n^2/(na-1)=[e_1,\dotsc,e_s]$. Then the mk1A $[e_1, \dotsc, e_{s-1}, \overline{e_s}]$ is an initial flipping mk1A with $\delta=n-a$. If $e_j = 2$ for all $i_0 < j \le s$ and $e_{i_0} \ge 3$, then the data for the flip $X^+$ is
\begin{equation*}
e_1 - [e_2, \dotsc, e_{i_0}-1].
\end{equation*}
\end{proposition}

\begin{proposition}\label{proposition:delta=1}
Let $(Q \in Y)$ be a cyclic quotient surface singularity. Suppose that there is an extremal $P$-resolution $X^+ \to Y$ defined by $[a_s, \dotsc, a_1]-c$ with $[a_1,\dotsc,a_s]=m^2/(ma-1)$. Suppose that $\delta=1$. Then $c=2$ and $a=m-1$. Hence the extremal $P$-resolution is of the form
\begin{equation*}
[m+2,\underbrace{2,\dotsc,2}_{m-2}]-2
\end{equation*}
and any mk1A or mk2A corresponding to $X^+ \to Y$ can be read from:
\begin{equation*}
[\underbrace{2,\dotsc,2}_{m-2},m+2]-[\underbrace{2,\dotsc,2}_{m-1}, \overline{m+3}]-\varnothing.
\end{equation*}
\end{proposition}

\begin{proof}
Since $\delta=cm_1'm_2'-m_1'a_2'-m_2'a_1'=(c-1)m-a=1$, we have $a=(c-1)m-1$. But $a < m$. Hence $c=2$ and $a=m-1$. Then the other assertions follow easily.
\end{proof}

\section{Simple embeddings and MMP}
\label{section:simple-embeddings-MMP}

Let $V$ be a (sufficiently small) regular neighborhood of a linear chain of negative $2$-spheres $\cup_{i=1}^{r} B_i$ with $B_i \cdot B_i \le -2$, which may be assumed to be diffeomorphic to the minimal resolution of a germ of a cyclic quotient surface singularity $P$. Assume that there is a smoothly embedded rational homology ball $B_{p,q} \hookrightarrow V$.

An embedding $B_{p,q} \hookrightarrow V$ is said to be \emph{simple} if the corresponding rational blow-up $W = (V \setminus B_{p,q}) \cup C_{p,q}$ is obtained by a sequence of ordinary blow-ups $\pi \colon W \to V$, that is, $W = V \sharp k \overline{\mathbb{CP}^2}$ for some $k \ge 1$;  Khodorovskiy~\cite{Khodorovskiy-2012,Khodorovskiy-2014}. But in this section we consider a slightly more restricted situation.

\begin{definition}\label{definition:plainly-simple}
A simple embedding $B_{p,q} \hookrightarrow V$ is said to be \emph{plainly simple} if the negative $2$-spheres of $C_{p,q}$ in $W$ consist of the proper transforms of $B_i$'s and that of the exceptional $2$-spheres of $\pi \colon W \to V$.
\end{definition}

\begin{remark}
It is not necessary that a simple embedding is plainly simple. Recently, Owens~\cite{Owens-2017,Owens-2019} show the existence of simple embeddings of $B_{p,q}$'s in $\mathbb{CP}^2$ which are not symplectic. For $B_{5,2}$ one can show that it is not plainly simple.
\end{remark}

In this section we show that every plainly simple embedding of $B_{p,q}$ in $V$ comes from a finite sequence of antiflips applied to a trivial family over a disk $\mathbb D$; Theorem~\ref{theorem:main-derived-from-MMP}. But, in order to state the theorem precisely, we need to introduce some notions.

From now on we assume that there is a plainly simple embedding of $B_{p,q}$ in $V$.

Let $V_{\mathbb{C}}$ be a complex model for $V$. That is, $V_{\mathbb{C}}$ is the minimal resolution of $P$ (of type $\frac{1}{m}(1,b)$) with the exceptional divisors $B=\sum_{i=1}^{r} B_i$ such that $(V, \cup B_i)$ is diffeomorphic to $(V_{\mathbb{C}}, \cup B_i)$.

Suppose that the dual graph of $B$ is given by
\begin{equation}\label{equation:[b]}
\begin{tikzpicture}
\node[bullet] (10) at (1,0) [label=above:{$-b_1$},label=below:{$B_1$}] {};
\node[bullet] (20) at (2,0) [label=above:{$-b_2$},label=below:{$B_2$}] {};

\node[empty] (250) at (2.5,0) [] {};
\node[empty] (30) at (3,0) [] {};

\node[bullet] (350) at (3.5,0) [label=above:{$-b_{r-1}$},label=below:{$B_{r-1}$}] {};
\node[bullet] (450) at (4.5,0) [label=above:{$-b_r$},label=below:{$B_r$}] {};

\draw [-] (10)--(20);
\draw [-] (20)--(250);
\draw [dotted] (20)--(350);
\draw [-] (30)--(350);
\draw [-] (350)--(450);
\end{tikzpicture}
\end{equation}
where $m/b = [b_1, \dotsc, b_r]$.

We may take a sequence of blow-ups $\pi_{\mathbb{C}} \colon W_{\mathbb{C}} \to V_{\mathbb{C}}$ so that the dual graph of $W_{\mathbb{C}}$ is the one from $W$, that is, $\pi_{\mathbb{C}} \colon W_{\mathbb{C}} \to V_{\mathbb{C}}$ is the complex model of the blow-ups $\pi \colon W \to V$.

So there is a diffeomorphism of pairs between $(C_{p,q} \subset W)$ and $(C_{p,q} \subset W_{\mathbb{C}})$. Since $(B_{p,q} \subset V)$ is obtained by rationally blow-down $W$ along $C_{p,q}$, there is a smoothly embedded $B_{p,q}$ in $V_{\mathbb{C}}$ induced by the rational blow-down of $C_{p,q}$ in $W_{\mathbb{C}}$ so that there is a diffeomorphism of pairs
\begin{equation*}
(B_{p,q} \subset V_{\mathbb{C}}) \sim (B_{p,q} \subset V).
\end{equation*}

We denote by $E=\sum_{k=1}^{t} E_k$ the divisor in $W_{\mathbb{C}}$ consisting of the rational curves corresponding to $C_{p,q}$ in $W_{\mathbb{C}}$. We contract $E$ from $W_{\mathbb{C}}$ so that we have a singular surface $X$ with one Wahl singularity $x \in X$. Let $f \colon (X \subset \mathcal{X}) \to (0 \in \mathbb{D})$ be a deformation of $X$ that is induced by the $\mathbb{Q}$-Gorenstein smoothing of $x \in X$. A general fiber $X_t = f^{-1}(t)$ ($t \neq 0$) is just the rational blow-down of $W_{\mathbb{C}}$ along $C_{p,q}$. Therefore we have a diffeomorphism of pairs
\begin{equation}\label{equation:diffeomorphism-of-pairs}
(B_{p,q} \subset X_t) \sim (B_{p,q} \subset V_{\mathbb{C}}) \sim (B_{p,q} \subset V).
\end{equation}

We now compactify $\mathcal{X}$ as we did in PPSU~\cite[\S3]{PPSU-2018} so that we will transform $f \colon \mathcal{X} \to \mathbb{D}$ into a proper map $\overline{f} \colon \overline{\mathcal{X}} \to \mathbb{D}$  because the relative MMP may be applied to proper maps mostly.

Let $\mathbb{F}_1$ be the Hirzebruch surface with the negative section $S_0$ with $S_0 \cdot S_0 = -1$. Let $S_{\infty}$ be another section of $\mathbb{F}_1$ with $S_{\infty} \cdot S_{\infty} = 1$ and let $F$ be a fiber of $\mathbb{F}_1$, that is, $F \cdot F = 0$. We blow up $\mathbb{F}_1$ appropriately at $S_0 \cap F$ (including infinitely near points over there) so that we obtain the following chain of rational curves starting from the proper transform $\overline{S}_0$ of $S_0$ and ending with the proper transform $\overline{S}_{\infty}$ of $S_{\infty}$:
\begin{equation}\label{equation:blowing-up-F_1}
\begin{tikzpicture}
\node[bullet] (10) at (1,0) [label=above:{$-b_1$},label=below:{$B_1=\overline{S}_0$}] {};
\node[bullet] (20) at (2,0) [label=above:{$-b_2$},label=below:{$B_2$}] {};

\node[empty] (250) at (2.5,0) [] {};
\node[empty] (30) at (3,0) [] {};

\node[bullet] (350) at (3.5,0) [label=above:{$-b_{r-1}$},label=below:{$B_{r-1}$}] {};
\node[bullet] (450) at (4.5,0) [label=above:{$-b_r$},label=below:{$B_r$}] {};

\node[bullet] (550) at (5.5,0) [label=above:{$-1$}] {};

\node[bullet] (650) at (6.5,0) [label=above:{$-a_e$},label=below:{$A_e$}] {};
\node[empty] (70) at (7,0) [] {};
\node[empty] (750) at (7.5,0) [] {};
\node[bullet] (80) at (8,0) [label=above:{$-a_2$},label=below:{$A_2$}] {};
\node[bullet] (90) at (9,0) [label=above:{$1-a_1$},label=below:{$A_1$}] {};
\node[bullet] (100) at (10,0) [label=above:{$+1$},label=below:{$\overline{S}_{\infty}$}] {};

\draw [-] (10)--(20);
\draw [-] (20)--(250);
\draw [dotted] (20)--(350);
\draw [-] (30)--(350);
\draw [-] (350)--(450);
\draw [-] (450)--(550);

\draw [-] (550)--(650);
\draw [-] (650)--(70);
\draw [dotted] (650)--(80);
\draw [-] (750)--(80);
\draw [-] (80)--(90);
\draw [-] (90)--(100);
\end{tikzpicture}
\end{equation}
where $m/(m-b)=[a_1, \dotsc, a_e]$. Let $\overline{V}_{\mathbb{C}}$ be the blown-up Hirzebruch surface $\mathbb{F}_1$ that contains the configuration of rational curves in Equation~\eqref{equation:blowing-up-F_1} above. Let $\overline{\pi}_{\mathbb{C}} \colon \overline{W}_{\mathbb{C}} \to \overline{V}_{\mathbb{C}}$ be the sequence of blow-ups that is the extension of $\pi_{\mathbb{C}} \colon W_{\mathbb{C}} \to V_{\mathbb{C}}$. We contract the divisor $E=\sum_{i=1}^{t} E_i$ in $\overline{W}_{\mathbb{C}}$ corresponding to $C_{p,q}$ so that we get a singular surface $\overline{X}$ with the Wahl singularity $x \in \overline{X}$. Note that
\begin{equation*}
X = \overline{X} \setminus (\cup_{i=1}^{e} A_i \cup \overline{S}_{\infty}).
\end{equation*}

We will show that the $\mathbb{Q}$-Gorenstein smoothing of $x \in \overline{X}$ can be extended to a deformation of $\overline{f} \colon \overline{\mathcal{X}} \to \mathbb{D}$ that preserves the curves $A_1, \dotsc, A_e$ and $\overline{S}_{\infty}$. For this, we contract again the linear chain $A=\cup_{j=1}^{e} A_j$ in $\overline{W}_{\mathbb{C}}$ and $\overline{X}$ to a singularity $Q$ respectively so that we get two singular surfaces $\widehat{W}_{\mathbb{C}}$ and $\widehat{X}$ with two singular points $y$ and $Q$.

\begin{lemma}\label{lemma:local-to-global-vanishes}
The obstruction to local-to-global deformations of $\widehat{X}$ vanishes. That is, $H^2(\widehat{X}, \sheaf{T}_{\widehat{X}})=0$, where $\sheaf{T}_{\widehat{X}}$ is the tangent sheaf of $\widehat{X}$.
\end{lemma}

\begin{proof}
The proofs in PPSU~\cite[Lemma~3.3, Proposition~3.4]{PPSU-2018} work verbatim.
\end{proof}

\begin{proposition}
Any smoothing $\mathcal{X} \to \mathbb D$ of $X$ can be extended to a smoothing $\overline{\mathcal{X}} \to \mathbb D$ of $\overline{X}$ which preserves the curves $A_1, \dotsc, A_e$ and $\overline{S}_{\infty}$ so that $X_t = \overline{X}_t \setminus (\cup_{i=1}^{e} A_i \cup \overline{S}_{\infty})$ for $0 \neq t \in \mathbb D$.
\end{proposition}

\begin{proof}
Since there is no local-to-global obstruction to deform $\widehat{X}$ by the above Lemma~\ref{lemma:local-to-global-vanishes}, one may take a deformation $\widehat{\mathcal{X}} \to \mathbb D$ of $\widehat{X}$ that is an extension of a smoothing of $Y$ such that it preserves the singular point $Q$. By taking a simultaneous resolution of the singularity $Q$ on each fibers of $\widehat{\mathcal{X}} \to \mathbb D$, we obtain the desired smoothing $\overline{\mathcal{X}} \to \mathbb D$ of $\overline{X}$.
\end{proof}

\begin{theorem}\label{theorem:main-derived-from-MMP}
Given a plainly simple embedding of a $B_{p,q}$ in $V$ as above, there is a sequence of flips $\mathcal{X}_1 \dashrightarrow \mathcal{X}_2 \dashrightarrow \dotsb \dashrightarrow \mathcal{X}_n$ with $\mathcal{X}=\mathcal{X}_1$ such that the central fiber $X_{n,0}$ of $\mathcal{X}_n$ is smooth and equal to $V_{\mathbb{C}}$. Hence we have that $(B_{p,q} \subset V)$ and $(B_{p,q} \subset X_{n,0})$ are diffeomorphic.
\end{theorem}

\begin{proof}
To simplify notation, let us write $\pi \colon W \to V$ instead of $\pi_{\mathbb{C}} \colon W_{\mathbb{C}} \to V_{\mathbb{C}}$. Since $b_2(V) < b_2(W)$, there are $(-1)$-curves on $W$. Moreover since the rational blow-down of $W$ along $C_{p,q}$ is diffeomorphic to $V$ (which is minimal), we must have $(-1)$-curves in $W$ intersecting $C_{p,q}$ at one point. Let $C$ be one of them. Since $C+C_{p,q}$ is negative definite (the whole exceptional divisor of $\pi$ plus the proper transform of the linear chain $B$ contracts to a cyclic quotient singularity), we have that $C+C_{p,q}$ contracts to a cyclic quotient singularity. As above, let $W \to X$ be the contraction of $C_{p,q}$, and let $(X \subset \mathcal{X}) \to (0 \in \mathbb{D})$ be a deformation of $X$ that is induced by the $\mathbb{Q}$-Gorenstein smoothing of the Wahl singularity $x \in X$. We have that  $K_{\mathcal{X}} \cdot C <0$. Then the blowing-down deformation of $C \subset X$ into a cyclic quotient singularity $y \in Y$ gives a flipping extremal neighborhood $\mathcal{X} \to \mathcal{Y}$ over $\mathbb{D}$. Note that a general fiber for both $\mathcal{X}$ and $\mathcal{Y}$ is diffeomorphic to $V$. Let $(C \subset \mathcal{X}) \dashrightarrow (C^+ \subset \mathcal{X}^+ =:\mathcal{X}_2)$ be the flip of $C$ over $\mathcal{Y} \to \mathbb{D}$.

We want to continue running the relative MMP. (We could think this is happening in families of projective surfaces over $\mathbb D$ by compactifying the surface $W$ as it was done in PPSU~\cite{PPSU-2018} and explained above. There are no local-to-global obstructions, and we run MMP relative to a blow-up of $B \subset V$.) For that, let us first assume that
$K_{\mathcal{X}^+}$ is nef. This is the same as assuming $K_{X^+}$ is nef, where $X^+$ is the central fiber (see Urz{\'u}a~\cite[Section 2]{Urzua-2016a}). Let $V \to v' \in V'$ be the contraction of $L$ to a cyclic quotient singularity $v' \in V'$. Then we have a contraction map $X^+ \to V'$, and it has $K_{X^+}$ nef, so it an M-resolution of $v' \in V'$ (i.e., only Wahl singularities and $K_{X^+}$ nef), this is a crepant partial resolution of a unique P-resolution of $V'$. In particular, the $\mathbb Q$-Gorenstein smoothing $X^+ \subset \mathcal{X}^+ \to \mathbb D$ blows-down to a deformation $V' \subset \mathcal{V}' \to \mathbb D$ so that the general fibers have the same Milnor numbers KSB~\cite{Kollar-Shepherd-Barron-1988}. Now, the general fiber of $X^+ \subset \mathcal{X}^+ \to \mathbb D$ is diffeomorphic to $V$, and so its Milnor number is the number of 2-spheres in the linear chain $B$. But then the deformation $V' \subset \mathcal{V}' \to \mathbb D$ happens in the Artin component of the versal deformation space of $v' \in V'$. This is because the Milnor number of a smoothing in the Artin component is the largest one from all components, since the difference of the dimension of any two components is equal to twice the difference between the corresponding Milnor numbers, and the largest dimension corresponds to the one of the Artin component Wahl~\cite[(4.7.3)]{Wahl-1981}. We note that the Milnor number for the Artin component is the number of 2-spheres in $B$. Therefore, the surface $X^+$ must be smooth, and so $X^+=V$, and this is what we wanted to conclude.

Otherwise, suppose that $K_{X^+}$ is not nef. Then as it was done e.g. in Urz{\'u}a~\cite[Section 2]{Urzua-2016a}, there must be a curve $C \subset X^+$ with $C \cdot K_{X^+}<0$ and $C^2 <0$, and so we are in the situation of an extremal neighborhood of type k1A or k2A as in HTU~\cite[Section 5]{HTU-2017}, i.e., induced from a $(-1)$-curve in the minimal resolution of $X^+$. It must be of flipping type since the general fiber of $X^+ \subset \mathcal{X}^+ \to \mathbb D$ is diffeomorphic to $V$. Let $(C \subset X^+ \subset \mathcal{X^+}) \dashrightarrow (C^+ \subset X_3 \subset \mathcal{X}_3)$ be the flip over $\mathbb D$. Note that again, by the same reasons as above, we have the contraction $X_3 \to V'$. And so, from this point we can apply induction on the number of flips.

We note that the process must end in a smooth deformation $\mathcal{X}_n \to \mathbb D$ for some $n$ since at each flip the index of the Wahl singularities involved strictly decrease (see e.g. HTU~\cite{HTU-2017}).

Since we used only flips, the general fiber $X_{n,t}$ ($t \neq 0$) is isomorphic to $V_{\mathbb C}$. Hence $X_{n,0}$ has no $(-1)$-curves, and since it is some blow-up of $V_{\mathbb{C}}$, we must have $X_{n,0}=V_{\mathbb{C}}$. Hence, referring to Equation~\eqref{equation:diffeomorphism-of-pairs}, MMP gives an explanation of how $B_{p,q}$ is simply embedded in $V$ via a trivial family plus antiMMP on its complex model $V_{\mathbb{C}}$.
\end{proof}

\begin{corollary}\label{corollary:No-in-An}
There is no plainly simple embedded rational homology ball $B_{p,q}$ in a neighborhood of the linear chain of $2$-spheres whose self-intersection numbers are all $(-2)$.
\end{corollary}

\begin{proof}
Suppose that there is a sequence of flips that induces a simple embedding of a rational homology ball. Think on the last flipped surface $X_n^+$. Then it comes from a $P$-resolution which consists of only one $(-2)$-curve. But a $(-2)$-curve is not $P$-resolution since canonical class is trivial on it.
\end{proof}

\begin{remark}
Khodorovskiy~\cite[Prop.5.1]{Khodorovskiy-2012} proves that there are no symplectically embedded $B_{n,1}$ in $E(2)$ for any $n \ge 2$, which implies that there is no symplectic embedding $B_{n,1} \hookrightarrow V_{-2}$ for any $n \ge 2$. The above corollary would be a generalization of this phenomenon.
\end{remark}

\section{Smoothly embedded rational homology balls via MMP}
\label{section:Rational-homology-balls}

We prove the existence of smoothly embedded rational homology balls using Mori sequences of antiflips. We first construct a special flipping mk1A associated to a given linear chain. Let $V$ be a germ of complex surface containing a linear chain
\begin{equation*}
\Gamma =
\begin{tikzpicture}[scale=0.75]
\node[bullet] (20) at (2,0) [label=above:{$-e_1$}] {};

\node[empty] (250) at (2.5,0) [] {};
\node[empty] (30) at (3,0) [] {};

\node[bullet] (350) at (3.5,0) [label=above:{$-e_t$}] {};

\draw [-] (20)--(250);
\draw [dotted] (20)--(350);
\draw [-] (30)--(350);

\end{tikzpicture}
\end{equation*}
of complex rational curves with $e_i \ge 2$ for all $i$. Assume that $e_t \ge 3$. If $n/a=[e_1,\dotsc,e_t-1]$, then, blowing up appropriately $V$ several times (cf. Corollary~\ref{corollray:[a]->[T]-[a]}), we have a regular neighborhood $U$ of the linear chain
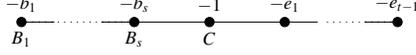
\begin{equation*}
\begin{tikzpicture}
\node[bullet] (20) at (2,0) [label=above:{$-b_1$},,label=below:{$B_1$}] {};

\node[empty] (250) at (2.5,0) [] {};
\node[empty] (30) at (3,0) [] {};

\node[bullet] (350) at (3.5,0) [label=above:{$-b_s$},label=below:{$B_s$}] {};
\node[bullet] (450) at (4.5,0) [label=above:{$-1$},label=below:{$C$}] {};

\node[bullet] (550) at (5.5,0) [label=above:{$-e_1$}] {};

\node[empty] (60) at (6,0) [] {};
\node[empty] (650) at (6.5,0) [] {};
\node[bullet] (70) at (7,0) [label=above:{$-e_{t-1}$}] {};

\draw [-] (20)--(250);
\draw [dotted] (20)--(350);
\draw [-] (30)--(350);
\draw [-] (350)--(450);
\draw [-] (450)--(550);

\draw [-] (550)--(60);
\draw [dotted] (60)--(70);
\draw [-] (650)--(70);
\end{tikzpicture}
\end{equation*}
where $n^2/(na-1)=[b_1, \dotsc, b_s]$. Let $U \to X$ be the contraction of the linear chain $B_1 \cup \dotsb \cup B_s$ in $U$ to a Wahl singularity $P \in X$ defined by the pair $(n,a)$. Denote again by $C \subset X$ the image of $C \subset U$. Let $U \to Z$ be the contraction of $C \subset X$ to a quotient singularity $Q \in Y$. Let $\mathcal{X} \to \mathbb{D}$ be a deformation of $X$ induced by the $\mathbb{Q}$-Gorenstein smoothing of the singularity $P \in X$ and let $\mathcal{Y} \to \mathbb{D}$ be the blown-down deformation of $\mathcal{X} \to \mathbb{D}$. According to Proposition~\ref{proposition:usual-flip}, $(C \subset \mathcal{X}) \to (Q \in \mathcal{Y})$ is an initial flipping mk1A.

\begin{definition}
The mk1A constructed above is called the \emph{usual initial flipping mk1A associated to the linear chain $\Gamma$}. We denote it
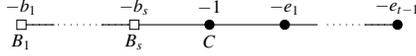
\begin{equation*}
\begin{tikzpicture}
\node[rectangle] (20) at (2,0) [label=above:{$-b_1$},,label=below:{$B_1$}] {};

\node[empty] (250) at (2.5,0) [] {};
\node[empty] (30) at (3,0) [] {};

\node[rectangle] (350) at (3.5,0) [label=above:{$-b_s$},label=below:{$B_s$}] {};
\node[bullet] (450) at (4.5,0) [label=above:{$-1$},label=below:{$C$}] {};

\node[bullet] (550) at (5.5,0) [label=above:{$-e_1$}] {};

\node[empty] (60) at (6,0) [] {};
\node[empty] (650) at (6.5,0) [] {};
\node[bullet] (70) at (7,0) [label=above:{$-e_{t-1}$}] {};

\draw [-] (20)--(250);
\draw [dotted] (20)--(350);
\draw [-] (30)--(350);
\draw [-] (350)--(450);
\draw [-] (450)--(550);

\draw [-] (550)--(60);
\draw [dotted] (60)--(70);
\draw [-] (650)--(70);
\end{tikzpicture}
\end{equation*}
where the vertex $\square$ implies that the corresponding rational curve is contracted to a Wahl singularity. For simplicity, we may denote the above mk2A by $[b_1, \dotsc, b_s]-1-e_1-\dotsb-e_{t-1}$, where the bracket $[\ ]$ implies that we contract them.
\end{definition}

\begin{remark}
The linear chain
\begin{tikzpicture}[scale=0.75]
\node[bullet] (20) at (2,0) [label=above:{$-e_1$}] {};

\node[empty] (250) at (2.5,0) [] {};
\node[empty] (30) at (3,0) [] {};

\node[bullet] (350) at (3.5,0) [label=above:{$-e_t$}] {};

\draw [-] (20)--(250);
\draw [dotted] (20)--(350);
\draw [-] (30)--(350);

\end{tikzpicture}
is the \emph{$\delta$-half linear chain} of the Wahl singularity $\frac{1}{n^2}(1,na-1)$ defined in PPS~\cite[Definition~2.6]{PPS-2016}.
\end{remark}

Almost all initial flipping mk1A associated to a linear chain have $\delta > 1$:

\begin{lemma}\label{lemma:almost-all-delta>1}
Let $\mathbb{E}_1$ be the usual initial flipping mk1A associated to $\Gamma$. Then $\delta=1$ if and only if $\Gamma$ is of the form
\begin{equation*}
\begin{tikzpicture}[scale=0.75]
\node[bullet] (20) at (2,0) [label=above:{$-2$}] {};

\node[empty] (250) at (2.5,0) [] {};
\node[empty] (30) at (3,0) [] {};

\node[bullet] (350) at (3.5,0) [label=above:{$-2$}] {};
\node[bullet] (450) at (4.5,0) [label=above:{$-3$}] {};

\draw [-] (20)--(250);
\draw [dotted] (20)--(350);
\draw [-] (30)--(350);
\draw [-] (350)--(450);
\end{tikzpicture}
\end{equation*}
\end{lemma}

\begin{proof}
By Proposition~\ref{proposition:usual-flip}, $\delta=n-a$. Therefore $n-a=1$ iff $a=n-1$ iff $[a_1,\dotsc,a_t-1]=n/(n-1)=[2,\dotsc,2]$.
\end{proof}

\begin{lemma}[{PPS~\cite[Corollary~3.12]{PPS-2016}}]
\label{lemma:flip->delta-half}
Applying flips described in Proposition~\ref{proposition:usual-flip} repeatedly to the usual initial flipping mk1A $(C \subset \mathcal{X}) \to (Q \in \mathcal{Y})$ associated to $\Gamma$, we have a deformation $\mathcal{Z} \to \mathbb{D}$ such that all of the fibers are smooth. In particular, the central fiber $Z_0$ of $\mathcal{Z} \to \mathbb{D}$ is the regular neighborhood $V$.
\end{lemma}

\begin{theorem}\label{theorem:main-linear}
Let $V$ be a regular neighborhood of a linear chain $\Gamma=E_1 \cup \dotsb \cup E_t$ ($t \ge 1$) of $2$-spheres whose dual graph is given by
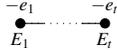
\begin{equation*}
\begin{tikzpicture}[scale=0.75]
\node[bullet] (20) at (2,0) [label=above:{$-e_1$},label=below:{$E_1$}] {};

\node[empty] (250) at (2.5,0) [] {};
\node[empty] (30) at (3,0) [] {};

\node[bullet] (350) at (3.5,0) [label=above:{$-e_t$},label=below:{$E_t$}] {};

\draw [-] (20)--(250);
\draw [dotted] (20)--(350);
\draw [-] (30)--(350);

\end{tikzpicture}
\end{equation*}
where $E_i \cdot E_i = -e_i \le -2$ for all $i$. Let $n/a=[e_1,\dotsc,e_t-1]$. If $e_t \ge 3$, then there are pairs $(B_{m_i, a_i}, B_{m_{i+1}, a_{i+1}})$ of disjoint rational homology balls smoothly embedded in $V$, where $(m_i, a_i)$ and $(m_{i+1}, a_{i+1})$ are consecutive pairs of Mori sequences associated to $n^2/(na-1)$. If, furthermore, $\Gamma$ is not of the form
\begin{tikzpicture}[scale=0.5]
\node[bullet] (20) at (2,0) [label=above:{$-2$}] {};

\node[empty] (250) at (2.5,0) [] {};
\node[empty] (30) at (3,0) [] {};

\node[bullet] (350) at (3.5,0) [label=above:{$-2$}] {};
\node[bullet] (450) at (4.5,0) [label=above:{$-3$}] {};

\draw [-] (20)--(250);
\draw [dotted] (20)--(350);
\draw [-] (30)--(350);
\draw [-] (350)--(450);
\end{tikzpicture}
then there are infinitely many such pairs of disjoint rational homology balls.
\end{theorem}

\begin{proof}
We may assume that $V$ itself is a complex surface containing the linear chain consisting of complex rational curves.

Let $\mathbb{E}_1: (C_1 \subset \mathcal{X}_1) \to (Q \in \mathcal{Y})$ be the usual initial flipping mk1A associated to $\Gamma$. According to Lemma~\ref{lemma:flip->delta-half}, after a sequence of finite flips applied to $(C_1 \subset \mathcal{X}_1) \to (Q \in \mathcal{Y})$, we have a deformation $\mathcal{Z} \to \mathbb{D}$ such that all of the fibers are smooth. Hence the central fiber $Z_0$ is diffeomorphic to a general fiber $Z_t$ ($t \neq 0$). Furthermore the central fiber $Z_0$ is equal to the regular neighborhood $V$. Therefore, since a flip changes only the central fiber, $X_{1,t}$ is diffeomorphic to $Z_t$, hence, to $V = Z_0$.

We may regard the mk1A $(C_1 \subset \mathcal{X}_1) \to (Q \in \mathcal{Y})$ as an initial flipping mk2A with one more Wahl singularity defined by the pair $(1,1)$. Let $(C_i \subset \mathcal{X}_i) \to (Q \in \mathcal{Y})$ ($i \ge 1$) be the mk2A with Wahl singularities defined by $(m_i, a_i)$ and $(m_{i+1}, a_{i+1})$ in the Mori sequence of mk2A's starting from $(C_1 \subset \mathcal{X}_1) \to (Q \in \mathcal{Y})$.

Notice that a general fiber $X_{i,t}$'s $(t \neq 0)$ of $\mathcal{X}_i \to \mathbb{D}$ contains disjoint rational homology balls $B_{m_i, a_i}$ and $B_{m_{i+1}, a_{i+1}}$ and two general fibers $X_{i,t}$ and $X_{j,t}$ for $t \neq 0$ are  diffeomorphic because this is a flip
over $\mathbb{D}$. Therefore every $X_{i,t}$ is diffeomorphic to $V$. Hence there are two disjoint pairs of rational homology balls $(B_{m_i, a_i}, B_{m_{i+1}, a_{i+1}})$ smoothly embedded in $V$, as asserted.

Furthermore, if $\Gamma$ is not of the form
\begin{tikzpicture}[scale=0.5]
\node[bullet] (20) at (2,0) [label=above:{$-2$}] {};

\node[empty] (250) at (2.5,0) [] {};
\node[empty] (30) at (3,0) [] {};

\node[bullet] (350) at (3.5,0) [label=above:{$-2$}] {};
\node[bullet] (450) at (4.5,0) [label=above:{$-3$}] {};

\draw [-] (20)--(250);
\draw [dotted] (20)--(350);
\draw [-] (30)--(350);
\draw [-] (350)--(450);
\end{tikzpicture}
then $\delta > 1$ by Lemma~\ref{lemma:almost-all-delta>1}. Hence the Mori sequence for $\mathbb{E}_1$ has infinitely many mk2A's, which implies that there are infinitely many pairs of disjoint rational homology balls embedded in $V$.
\end{proof}

\begin{proposition}\label{proposition:simple-embedding}
The embeddings in the above Theorem~\ref{theorem:main-linear} are simple.
\end{proposition}

\begin{proof}
Let $\mathbb{E}_i: (C_i \subset \mathcal{X}_i) \to (Q \in \mathcal{Y})$ be the $i$-th flipping mk2A constructed in the proof of Theorem~\ref{theorem:main-linear}. We showed in the proof that a general fiber $X_{i,t}$ is diffeomorphic to the regular neighborhood $V$. Note that the rationally blown up $Z$ can be obtained from $V$ by a sequence of ordinary blowing ups. Hence the assertion follows.
\end{proof}

\begin{example}
The usual initial flipping mk2A $\mathbb{E}_1$ associated to
\begin{tikzpicture}[scale=0.5]
\node[bullet] (350) at (3.5,0) [label=above:{$-3$}] {};
\node[bullet] (450) at (4.5,0) [label=above:{$-3$}] {};
\draw [-] (350)--(450);
\end{tikzpicture}
is
\begin{tikzpicture}[scale=0.5]
\node[rectangle] (10) at (1,0) [label=above:{$-3$}] {};
\node[rectangle] (20) at (2,0) [label=above:{$-5$}] {};
\node[rectangle] (30) at (3,0) [label=above:{$-2$}] {};
\node[bullet] (40) at (4,0) [label=above:{$-1$}] {};
\node[bullet] (50) at (5,0) [label=above:{$-3$}] {};
\draw [-] (10)--(20);
\draw [-] (20)--(30);
\draw [-] (30)--(40);
\draw [-] (40)--(50);
\end{tikzpicture}
For example, the second mk2A $\mathbb{E}_2$ is
\begin{tikzpicture}[scale=0.5]
\node[rectangle] (10) at (1,0) [label=above:{$-3$}] {};
\node[rectangle] (20) at (2,0) [label=above:{$-7$}] {};
\node[rectangle] (30) at (3,0) [label=above:{$-2$}] {};
\node[rectangle] (40) at (4,0) [label=above:{$-2$}] {};
\node[rectangle] (50) at (5,0) [label=above:{$-3$}] {};
\node[rectangle] (60) at (6,0) [label=above:{$-2$}] {};
\node[bullet] (70) at (7,0) [label=above:{$-1$}] {};
\node[rectangle] (80) at (8,0) [label=above:{$-3$}] {};
\node[rectangle] (90) at (9,0) [label=above:{$-5$}] {};
\node[rectangle] (100) at (10,0) [label=above:{$-2$}] {};
\node[bullet] (110) at (11,0) [label=above:{$-3$}] {};
\draw [-] (10)--(20);
\draw [-] (20)--(30);
\draw [-] (30)--(40);
\draw [-] (40)--(50);
\draw [-] (50)--(60);
\draw [-] (60)--(70);
\draw [-] (70)--(80);
\draw [-] (80)--(90);
\draw [-] (90)--(100);
\draw [-] (100)--(110);
\end{tikzpicture}.
The first flip on $\mathbb{E}_1$ is given by
\begin{tikzpicture}[scale=0.5]
\node[bullet] (10) at (1,0) [label=above:{$-3$}] {};
\node[rectangle] (20) at (2,0) [label=above:{$-4$}] {};
\node[bullet] (30) at (3,0) [label=above:{$-1$}] {};
\draw [-] (10)--(20);
\draw [-] (20)--(30);
\end{tikzpicture}.
Then, the second flip induces the original linear chain
\begin{tikzpicture}[scale=0.5]
\node[bullet] (10) at (1,0) [label=above:{$-3$}] {};
\node[bullet] (20) at (2,0) [label=above:{$-3$}] {};
\draw [-] (10)--(20);
\end{tikzpicture}.
\end{example}

\begin{remark}
One may construct other initial flipping mk2A's $\mathbb{E}_1$ $(C \subset \mathcal{X}_1)$ associated to the given linear chain $\Gamma$ such that there is a sequence of flips $\mathcal{X}_1 \overset{flip}{\dashrightarrow} \mathcal{X}_2 \overset{flip}{\dashrightarrow} \dotsb \overset{flip}{\dashrightarrow} \mathcal{X}_n$ where the central fiber $X_{n,0}$ of $\mathcal{X}_n$ is the regular neighborhood of the given linear chain $\Gamma$. For example, the linear chain
\begin{tikzpicture}[scale=0.5]
\node[bullet] (10) at (1,0) [label=above:{$-3$}] {};
\node[bullet] (20) at (2,0) [label=above:{$-2$}] {};
\node[bullet] (30) at (3,0) [label=above:{$-4$}] {};
\draw [-] (10)--(20);
\draw [-] (20)--(30);
\end{tikzpicture}
has the usual one $[3,2,5,4,2]-1-3-2$ but also the new one $1-[4,2,5,4,2,2]-1-[4]-2$. So there are another Mori sequences associated to the given linear chain.
\end{remark}

\begin{remark}
If \(\Gamma =
\begin{tikzpicture}[scale=0.5]
\node[bullet] (20) at (2,0) [label=above:{$-2$}] {};

\node[empty] (250) at (2.5,0) [] {};
\node[empty] (30) at (3,0) [] {};

\node[bullet] (350) at (3.5,0) [label=above:{$-2$}] {};
\node[bullet] (450) at (4.5,0) [label=above:{$-3$}] {};

\draw [-] (20)--(250);
\draw [dotted] (20)--(350);
\draw [-] (30)--(350);
\draw [-] (350)--(450);
\end{tikzpicture}\),
then $B_{n,1}$ and $B_{n+1,1}$ are smoothly embedded in $V$ by Proposition~\ref{proposition:delta=1}. We do not know whether there are more rational homology balls in $V$.
\end{remark}

The following Corollary~\ref{corollary:V_-n} and Theorem~\ref{theorem:main-sharp} generalize Khodorovskiy~\cite[Theorems~1.2,~1.3]{Khodorovskiy-2014}.

\begin{corollary}\label{corollary:V_-n}
Let $V_{-n}$ be a neighborhood of a $2$-sphere with self-intersection number $(-n)$ for $n \ge 4$. Then there exist infinitely many pairs of disjoint rational homology balls smoothly embedded in $V_{-n}$. In particular, we have $B_{n-1,1} \hookrightarrow V_{-n}$. Furthermore, if $n=4$, then there is an embedding $B_{2m+1,1} \hookrightarrow V_{-4}$ for any $m \ge 1$.
\end{corollary}

\begin{proof}
The numerical data for the initial mk1A associated to
\begin{tikzpicture}[scale=0.75]
\node[bullet] (10) at (1,0) [label=above:{$-n$}] {};
\end{tikzpicture}
is $[(n-1, 1)]-1-[(1,1)]$, that is, $[n+1,2,\dotsc,\overline{2}]$. So there is an embedding $B_{n-1,1} \hookrightarrow V_{-n}$. If $n=4$, then the numerical data for the initial mk1A is $[5,\overline{2}]$. So the Mori sequence can be described simply by
\begin{equation*}
\dotsb - [2m+3,\overline{2},2,\dotsc,2] - \dotsb - [5,\overline{2}] - \varnothing
\end{equation*}
where $m \ge 1$. Therefore $B_{2m+1,1}$ $(m \ge 1)$ is smoothly embedded in $V_{-4}$.
\end{proof}

\begin{theorem}\label{theorem:main-sharp}
For any integers $n, a$ with $n > a \ge 1$ and $(n,a)=1$, there are infinitely many pairs of disjoint rational homology balls $(B_{m_i, a_i}, B_{m_{i+1}, a_{i+1}})$ smoothly embedded in $B_{n,a} \sharp \overline{\mathbb{CP}^2}$, where $(m_i, a_i)$ and $(m_{i+1}, a_{i+1})$ are consecutive pairs of Mori sequences associated to the initial divisorial mk2A with two Wahl singularities defined by the pairs $(n,a)$ and $(n^2, na-1)$. In particular, $(m_1, a_1)=(n, a)$ and $(m_2, a_2)=(n^2, n^2-(na-1))$, that is, $B_{n,a} \sqcup B_{n^2,na-1} \hookrightarrow B_{n,a} \sharp \overline{\mathbb{CP}^2}$. In particular, there is an embedding $B_{2m,1} \hookrightarrow B_{2,1} \sharp \overline{\mathbb{CP}}^2$ for any $m \ge 1$.
\end{theorem}

\begin{proof}
Let $(Q \in Y)$ be the Wahl singularity defined by the pair $(n,a)$. Let $V \to Y$ be the minimal resolution of $Q$ whose dual graph is given by
\begin{equation*}
\begin{tikzpicture}[scale=0.75]
\node[bullet] (20) at (2,0) [label=above:{$-a_1$}] {};

\node[empty] (250) at (2.5,0) [] {};
\node[empty] (30) at (3,0) [] {};

\node[bullet] (350) at (3.5,0) [label=above:{$-a_s$}] {};

\draw [-] (20)--(250);
\draw [dotted] (20)--(350);
\draw [-] (30)--(350);

\end{tikzpicture}
\end{equation*}
with $n^2/(na-1)=[a_1, \dotsc, a_s]$. By Corollary~\ref{corollray:[a]->[T]-[a]}, after a sequence of appropriate blowing ups of $V$, we have a regular neighborhood $\widetilde{V}$ of the linear chain
\begin{equation*}
\begin{tikzpicture}[scale=0.75]
\node[bullet] (20) at (2,0) [label=above:{$-b_1$}] {};

\node[empty] (250) at (2.5,0) [] {};
\node[empty] (30) at (3,0) [] {};

\node[bullet] (350) at (3.5,0) [label=above:{$-b_t$}] {};
\node[bullet] (450) at (4.5,0) [label=above:{$-1$},label=below:{C}] {};

\node[bullet] (550) at (5.5,0) [label=above:{$-a_1$}] {};

\node[empty] (60) at (6,0) [] {};
\node[empty] (650) at (6.5,0) [] {};

\node[bullet] (70) at (7,0) [label=above:{$-a_s$}] {};

\draw [-] (20)--(250);
\draw [dotted] (20)--(350);
\draw [-] (30)--(350);

\draw [-] (350)--(450);
\draw [-] (450)--(550);

\draw [-] (550)--(60);
\draw [dotted] (550)--(70);
\draw [-] (650)--(70);
\end{tikzpicture}
\end{equation*}
with $(n^2)^2/(n^2(na-1)-1)=[b_1, \dotsc, b_t]$. Then we have an mk2A $\mathbb{E}_1: (C \subset \mathcal{X}) \to (Q \in \mathcal{Y})$ where there are two Wahl singularities on $C$ defined by the pairs $(n^2, na-1)$ and $(n,a)$. Then $\delta=n > 1$ and one can easily check that this mk2A $\mathbb{E}_1$ is initial and of divisorial type.

Let $\mathbb{E}_i: (C_i \subset \mathcal{X}_i) \to (Q \in \mathcal{Y})$ be the mk2A in the Mori sequence for $\mathbb{E}_1$. By Proposition~\ref{proposition:divisorial-contraction-for-mk12A}, there is a blow-down $X_{i,t} \to Y_t$ ($t \neq 0$) of a $(-1)$-curve between two smooth fibers. Note that $Y_t=B_{n,a}$ and there are two disjoint rational homology balls in $X_{i,t}$.

If $n=2$ and $a=1$, then the Mori sequence is just
\begin{equation*}
\dotsb - [2m+2,\overline{2},2,\dotsc,2] - \dotsb - [6,\overline{2},2]-[4]
\end{equation*}
where $m \ge 1$. Therefore $B_{2m,1}$ $(m \ge 1)$ is smoothly embedded in $B_{2,1} \sharp \overline{\mathbb{CP}}^2$. Hence the assertion follows.
\end{proof}

The above theorem generalizes Owens~\cite[Theorem~4]{Owens-2017} which asserts that there is an embedding $B_{n^2, n-1} \hookrightarrow B_{n,1} \sharp \overline{\mathbb{CP}^2}$.

On the other hand, Owens~\cite[\S5]{Owens-2017} extends the notion of simple to embeddings of the form $B_{p,q} \hookrightarrow B_{p', q'} \sharp \overline{\mathbb{CP}^2}$ as follows: The embedding $B_{p,q} \hookrightarrow B_{p', q'} \sharp \overline{\mathbb{CP}^2}$ is said to be \emph{simple} if the resulting rational blow-up of $B_{p,q}$ has the same effect as rationally blowing up $B_{p',q'}$, together with a sequence of ordinary blow-ups.

\begin{proposition}\label{proposition:simple-embedding-sharp}
The embeddings $B_{m_i, a_i} \hookrightarrow B_{n,a} \sharp \overline{\mathbb{CP}^2}$ in the above Theorem~\ref{theorem:main-sharp} are simple for $i > 1$.
\end{proposition}

\begin{proof}
As in the proof of Theorem~\ref{theorem:main-sharp}, let $\mathbb{E}_i: (C_i \subset \mathcal{X}_i) \to (Q \in \mathcal{Y})$ ($i \ge 2$) be the divisorial mk2A with two Wahl singularities $P_1$ and $P_2$ defined by the pairs $(m_i, a_i)$ and $(m_{i+1}, a_{i+1})$ respectively over $(Q \in \mathcal{Y})$ where $Q \in Y_0$ is a Wahl singularity defined by the pair $(n,a)$. If we smooth out $P_1 \in C_i \subset X_{i,0}$ but keep $P_2$, then the mk2A $\mathbb{E}_i$ deforms to a divisorial mk1A $\mathbb{E}_i': (C_i' \subset \mathcal{X}_i') \to (Q \in \mathcal{Y})$ with one Wahl singularity defined by the pair $(m_i, a_i)$. Then the two embeddings $B_{m_i, a_i} \hookrightarrow B_{n,a} \sharp \overline{\mathbb{CP}^2}$ via the mk2A $\mathbb{E}_i$ and the mk1A $\mathbb{E}_i'$ are equivalent to each other.

Let $B_{m_i, a_i} \hookrightarrow B_{n,a} \sharp \overline{\mathbb{CP}^2}$ be the embedding obtained via the mk1A $\mathbb{E}_i'$. The rational blow-up along $B_{m_i, a_i}$ is equivalent to taking the minimal resolution of $P_i \in X_{i,0}'$ on the central fiber $X_{i,0}'$ of $\mathbb{E}_i'$, which is equivalent to taking the minimal resolution $Q \in Y_0$ followed by a sequence of ordinary blow-ups. So the assertion follows.
\end{proof}

\section{Rational homology balls in Milnor fibers}
\label{section:other-embedding}

We show that there are infinitely many rational homology balls smoothly embedded in the Milnor fibers of certain cyclic quotient surface singularities and the embeddings are non-simple. Thus there is also little obstruction to embedding rational homology balls in smooth $4$-manifolds in a `non-simple' way.

\begin{theorem}\label{theorem:main-2-handlebody}
Let $(Q \in Y)$ be a cyclic quotient surface singularity which admits an extremal P-resolution $f^+ \colon X^+ \to Y$ with one Wahl singularity defined by the pair $(m', a')$. Then there are infinitely many pairs of disjoint rational homology balls smoothly embedded in the Milnor fiber $M$ associated to the P-resolution $X^+ \to Y$ and all of these embeddings are non-simple.
\end{theorem}

\begin{proof}
Let $(-c)$ be the self-intersection number of the strict transform of the exceptional divisor of $f^+$ in the minimal resolution of $X^+$. We define an mk2A $\mathbb{E}_1: (C_1 \subset \mathcal{X}_1) \to (Q \in \mathcal{Y})$ by the numerical data $(m_1, a_1)=(1,1)$ and $(m_2, a_2)=(cm'-a', m'-a')$. Then $\delta=cm'-m'-a'<0$ (for $c \ge 2$ and $m'>a'$) and $\delta m_1 - m_2 = -m' < 0$. Therefore $\mathbb{E}_1$ is initial and flipping such that, on its flip $(C_1 \subset \mathcal{X}_1) \dashrightarrow (C^+ \subset \mathcal{X}^+)$, the extremal $P$-resolution $X^+$ is the central fiber of $\mathcal{X}^+ \to \mathbb{D}$.

Notice that there are infinitely many entries $\mathbb{E}_i: (C_i \subset \mathcal{X}_i) \to (Q \in \mathcal{Y})$ ($i \ge 1$) of mk2A's in the Mori sequence starting $\mathbb{E}_1$ and that $M \cong X_t^+ \cong X_{i,t}$ for $0 \neq t \in \mathbb{D}$. Hence there are infinitely many pairs of disjoint rational homology balls induced from $\mathbb{E}_i$ for $i \ge 1$ smoothly embedded in $M$.

Suppose that an embedding $B \hookrightarrow M$ is induced from an mk2A $\mathbb{E}: (C \subset \mathcal{X}) \to (Q \in \mathcal{Y})$ in the Mori sequence starting from $\mathbb{E}_1$. As in the proof of Proposition~\ref{proposition:simple-embedding-sharp}, the embedding $B \hookrightarrow M$ is indeed induced from an mk1A $\mathbb{E}': (C' \subset \mathcal{X}') \to (Q \in \mathcal{Y})$ obtained by deforming the mk2A $\mathbb{E}$ by Urz{\'u}a~\cite[Proposition 2.12]{Urzua-2016}.

Let $Z$ be the rationally blown-up of the $4$-manifold from $M$ along $B$. Then, by construction, $Z$ is diffeomorphic to the minimal resolution of the central fiber $X_0'$ of $\mathcal{X}' \to \mathbb{D}$ (where $K_{X'_0}$ is negative), which can be obtained from the minimal resolution of $Y$ by a sequence of ordinary blow-ups. Therefore $M$ must be diffeomorphic to the minimal resolution of $Y$, but then $M$ is the Milnor fiber of the Artin component, and so the P-resolution is smooth, a contradiction.
\end{proof}


\end{document}